\documentclass[11pt]{article}

\usepackage{amsmath}
\usepackage{amsthm}
\usepackage{amssymb}
\usepackage{fullpage}
\usepackage{verbatim}

\newtheorem{theorem}{Theorem}[section]
\newtheorem{lemma}[theorem]{Lemma}
\newtheorem{proposition}[theorem]{Proposition}
\newcounter{mirrorcount}
\newtheorem{mirrorprop}{Theorem}[mirrorcount]

\theoremstyle{definition}

\theoremstyle{remark}
\newtheorem{remark}[theorem]{Remark}
\theoremstyle{remark}

\newcommand\nc\newcommand
\newcommand\Zmodp{\mathbb{Z}/p\mathbb Z}
\newcommand\Zmodq{\mathbb{Z}/q\mathbb Z}
\newcommand\mapq{\phi_q}
\newcommand\Mapq{\Phi_q}
\newcommand\quotientmap{\phi_p}
\newcommand\qmap{\quotientmap}
\newcommand\Qmap{\Phi_p}
\nc\Zr{\Z[y_1,\ldots,y_{j+1}]}

\nc\dsp\displaystyle
\nc\reason[1]{\mbox{\parbox{.25\linewidth}{(#1)}}}
\nc\smlreason[2]{\mbox{\parbox{#1}{\small (#2)}}}
\nc\npreason[1]{\mbox{(#1)}}
\newenvironment{romanlist}%
    {\begin{list}{(\roman{enumi})\ }{\usecounter{enumi}}}%
    {\end{list}}
\nc\ii\item
\newcommand\abs[1]{\left|#1\right|}
\newcommand\parens[1]{\left(#1\right)}
\newcommand\cbraces[1]{\left\{#1\right\}}

\newcommand\angles[1]{\left\langle #1 \right\rangle}
\newcommand\cross{\times}

\newcommand\isom{\simeq}
\newcommand{\C}{\mathbb{C}}
\newcommand{\Z}{\mathbb{Z}}
\newcommand{\Q}{\mathbb{Q}}
\newcommand{\I}{\mathbf{I}}
\newcommand\Qbar{\overline{\Q}}

\nc\bb\mathbb
\nc\lt\left
\nc\rt\right
\newcommand\e[1]{\begin{align*} #1
\end{align*}}
\newcommand\en[1]{\begin{align} #1
\end{align}}
\newcommand\mc\mathcal
\newcommand\fk\mathfrak
\nc\czid{characteristic zero integral domain}
\nc\SL{\operatorname{SL}}
\nc\SLt{\operatorname{SL_2}}
\nc\SLtp{\operatorname{SL}_2(\Zmodp)}
\nc\uA{\underline{A}}
\nc\uG{\underline{G}}
\nc\uN{\underline{N}}

\title{Mapping Incidences}

\renewcommand{\footnotemark}{}

\author{
\parbox{6in}{
    \begin{center}
    Van H. Vu
	\\
    \small Department of Mathematics, Rutgers University, Piscataway, NJ
08854, USA
    \\ \texttt{vanvu@math.rutgers.edu}
    \end{center}
}
\\
\parbox{6in}{
    \begin{center}
    Melanie Matchett Wood \\
    \small American Institute of Mathematics, 360 Portage Ave, Palo Alto, CA
94304, USA \\
    Stanford University, Dept of Mathematics, Building 380, Sloan Hall,
Stanford, CA 94305, USA 
    \\ \texttt{mwood@math.stanford.edu}
    \end{center}
}
\\
\parbox{6in}{
    \begin{center}
    Philip Matchett Wood \\
    \small 
    Stanford University, Dept of Mathematics, Building 380, Sloan Hall,
Stanford, CA 94305, USA 
    \\ \texttt{pmwood@math.stanford.edu}
    \end{center}
}
\footnote
{\emph{Mathematics Subject Classification}: 05B25 (primary), 11T99 (secondary).} 
\footnote{The first author is supported by research grants DMS-0901216 and
AFOSAR-FA-9550-09-1-0167.  The second and third authors were supported by
National Defense Science and Engineering Graduate Fellowships and by National
Science Foundation Graduate Research Fellowships.}
}

\begin{document}
\maketitle

\begin{abstract}
We show that  any finite set $S$ in a \czid\ can be mapped to
$\Zmodp$, for infinitely many primes $p$, preserving all algebraic
incidences in $S$. This can be seen as a generalization of the
well-known Freiman isomorphism lemma, which asserts that any finite subset
of a torsion-free group can be mapped into $\Zmodp$, preserving all
linear incidences.

As applications, we derive several combinatorial results (such as
sum-product estimates) for a finite set in a \czid. As $\bb C$ is a
\czid, this allows us to obtain new proofs for some recent results
concerning finite sets of complex numbers, without relying on the
topology of the plane.
\end{abstract}

\section {Introduction}

Many problems and results in arithmetic combinatorics deal with
algebraic incidences in a finite set $S$. Classical examples are
the Szemer\' edi-Trotter theorem, and sum-product estimates.

A  well-studied situation is when $S$ is a subset of $\Zmodp$, the
finite field with $p$ elements where $p$ is a large prime. In this
case, the special structure of the field and  powerful techniques
such as discrete Fourier analysis provide many tools to attack these
problems. These  features are not available in other settings and it
seems one needs to invent new tricks. For example, when $S$ is a
subset of the complex numbers, most
 studies previous to this paper relied on  some very clever use of properties of the
plane. Thus, it seems desirable to have a  tool that reduces a
problem from a general setting to the special case of $\Zmodp$.

Such a tool exists, if one only cares about the linear relations
among the elements of $S$. In this case,  the famous Freiman
isomorphism lemma (see, for example, \cite[Lemma 5.25]{TVbook})
asserts that any finite subset of an arbitrary torsion-free group
can be mapped into $\Zmodp$, given that $p$ is sufficiently large,
preserving all additive (linear) relations in $S$. Thanks to this
result, it has now become a  common practice in additive
combinatorics to translate additive problems in general torsion-free
groups to corresponding problems in $\Zmodp$.

The goal of this paper is to show that the desired reduction is
possible in general. Technically speaking, we prove that any finite
set $S$ in a \czid\ can be mapped to $\Zmodp$, for infinitely many
primes $p$, preserving all algebraic incidences in $S$.

Some notable \czid s include the integers, the complex numbers, and
the field of rational functions $\bb C(t_1,t_2,\ldots)$ in any
number of formal variables $t_i$. As applications, we obtain some
new  results and short proofs of some known results. In particular,
it is shown that sum-product estimates and bounds for incidence
geometry problems over $\Zmodp$ imply the same bounds for the
analogous problems over any \czid\ (including the real and complex
numbers).

Throughout this paper, we assume that all rings are commutative with identity
$1$ and that all ring homomorphisms take $1$ to $1$.  Let $D$ be a
characteristic zero integral domain (so $D$ is a commutative ring with
identity that has no zero divisors).  We will identify the subring of $D$
generated by the identity with the integers $\Z$ (since the two are
isomorphic).  For a subset $S$ of $D$, we will use $\bb Z[S]$ to denote the
smallest subring of $D$ containing $S$.

\begin{theorem}
\label{redu2}  Let $S$ be a finite subset of
a characteristic zero integral domain $D$, and let $L$ be a finite set of
non-zero elements in the subring $\Z[S]$ of $D$.  There exists an infinite
sequence of primes with positive relative density such that for each prime $p$
in the sequence, there is a ring homomorphism $\quotientmap: \Z[S]\to\Zmodp$
satisfying $0\notin \quotientmap(L)$.
\end{theorem}

By \emph{positive relative density}, we mean that the sequence has positive
density in the sequence of all primes. It is important to note that
Theorem~\ref{redu2} is not true for all primes. For example, if $S=\{i\}
\subset \bb C$ and $L$ is arbitrary, then the desired map does not exist for $p = -1 \pmod{4}$,
since the equation $x^2 =-1$ is not solvable in $\Zmodp$ for
these $p$.
Note that for the applications of Theorem~\ref{redu2} in this paper, we only
need that there exist infinitely many primes such that a map $\qmap$ exists,
which follows from those primes having positive relative density.

The role of $L$ in Theorem~\ref{redu2} is to guarantee that the homomorphism
$\quotientmap$ is injective on certain subsets of $\Z[S]$.  Such injectivity
is often necessary when applying Theorem~\ref{redu2}; for example, if one
were interested in the cardinality of $S$, one could guarantee that
$\quotientmap$ is injective on $S$ (and thus preserves the cardinality of
$S$) by setting $L:=\{s_1-s_2: s_1,s_2 \in S\}$.

Theorem~\ref{redu2} does not give upper bounds on the sizes of the smallest
primes $p$ in the infinite sequence it produces.  It would be an interesting
question to study whether a version of Theorem~\ref{redu2} can be proven that
includes, for example, an upper bound for at least one prime in the infinite
sequence, where the bound would depend on both $S$ and $L$ (see
Remark~\ref{RemEffective}).  Another interesting question is the following:
Given a set $A \subset \Zmodp$, are there conditions on $A$ and $p$ (say, that
$A$ is very small with respect to $p$) that allow one to construct a map that preserves algebraic incidences and that sends $A$ into some \czid\ (for example, $\bb Z$)? 

Readers interested in the methods of the current paper may also be interested
in the lecture by Serre \cite{Serre2009} (posted on the Math ArXiv) titled
``How to use finite fields for problems concerning infinite fields,'' which
focuses on problems in algebraic geometry.  An excellent discussion of Serre's
lecture from a general mathematical viewpoint may be found on Terence Tao's
blog \cite{TTaoBlog2009}, and Tao also mentions some relations between Serre's
lecture and the current paper.

This paper is organized as follows.  In the next few sections, we present a few
sample applications of Theorem~\ref{redu2}. Combining arguments from
\cite{BKT} with Theorem~\ref{redu2}, we prove 
a Szemer\' edi-Trotter-type result in Section~\ref{s:StTr}.
In Section~\ref{s:sum
prod}, we use Theorem~\ref{redu2} to demonstrate a sum-product estimate for
\czid s, based on well-known sum-product estimates in $\Zmodp$.
Section~\ref{s:Helf} is focused on combining a product result for $\SLtp$ from
\cite{Helf} with Theorem~\ref{redu2} to get an analogous product result for
$\SLt(D)$, where $D$ is a \czid.  In Section~\ref{s:mat}, we show that a
random matrix taking finitely many values in a \czid\ is singular with
exponentially small probability. This extends earlier results on integer
matrices to the complex setting. Finally, the proof of Theorem~\ref{redu2} is
given in Section~\ref{s:proofs}.

\section{A Szemer\' edi-Trotter-type result for \czid s}\label{s:StTr}

In this section, we apply Theorem~\ref{redu2} to the problem of bounding the
maximum number of incidences between a finite set of lines and a finite set of
points.  The well-known Szemer\' edi-Trotter Theorem \cite{SzTr} solves this
problem in the case of points and lines in $\bb R \cross \bb R$.
Recently, in \cite{BKT}, an analogous result was proven
for $\Zmodq\cross \Zmodq$ where $q$ is a prime.

\begin{theorem}[(Theorem~6.2 in {\cite{BKT}})]\label{t:bktSzTr}
Let $q$ be a prime, and let $\mc P$ and $\mc L$ be sets of points and lines,
respectively, in $\Zmodq \cross \Zmodq$ such that the cardinalities $\abs {\mc
P},\abs{\mc L}\le N \le q$.  Then there exist positive absolute constants $c$
and $\delta$ such that
\en{\label{bktSzTr}
\abs{\{(p,\ell)\in \mc P\cross \mc L: p \in \ell\}\rule{0pt}{12pt}} \le c N^{3/2
-\delta}.
}
\end{theorem}

\begin{remark}\label{SzTrRem} The original version of Theorem~\ref{t:bktSzTr}
proven in \cite{BKT} relied on the best known sum-product result at
the time (also found in \cite{BKT}), which worked only for subsets
of $\Zmodq$ with cardinality between $q^\alpha$ and $q^{1-\alpha}$ for
a constant $\alpha$.  In particular, the proof in \cite{BKT} assumed
that Inequality~\eqref{bktSzTr} was false and used this assumption
to construct a subset $A$ of $\Zmodq$ with cardinality
$N^{1/2-C\delta}$, for some constant $C$, such that
$\max\{\abs{A+A},\abs{AA}\}$ was small, a contradiction of the
sum-product estimate proven in \cite{BKT}.  Thus, the version of
Theorem~\ref{t:bktSzTr} in \cite{BKT} required the additional
assumption that $N = q^\alpha$ for a constant $\alpha$.

To prove  Theorem~\ref{t:bktSzTr}  as stated above,  one can simply
replace the sum-product results in \cite{BKT} by  more recent
estimates that apply for all subsets of $\Zmodq$ (for example,
\cite{Konyagin,Garaev,KS}).
\end{remark}

In a general ring $R$, we define a line to be the set of solutions $(x,y)$ in
$R\cross R$ to an equation $y = mx+b$, where $m$ and $b$ are fixed elements of
$R$.
Using Theorem~\ref{redu2}, we prove  that the same bound
as in Theorem~\ref{t:bktSzTr} holds for an arbitrary \czid:

\begin{theorem}\label{ourSzTr}
Let $D$ be a \czid, and let $\mc P$ and $\mc L$ be sets of points and lines
(respectively) in $D \cross D$ with cardinalities $\abs{\mc P},\abs{\mc L}\le
N$.  Then there exist positive absolute constants $c$ and $\delta$ such that
$$\abs{\{(p,\ell)\in \mc P\cross \mc L: p \in \ell\}\rule{0pt}{12pt} }\le c N^{3/2
-\delta}.$$
\end{theorem}
\noindent The constants $c$ and $\delta$ are the same as those in
Theorem~\ref{t:bktSzTr}.  Any improvement to Theorem~\ref{t:bktSzTr}, for
example, better constants or giving a good bound when $\mc P$ and $\mc L$ have
very different cardinalities, would also immediately translate to
Theorem~\ref{ourSzTr} above.  In the case of $\bb R\cross \bb R$, this theorem
is true with $\delta$ being replaced with the optimal constant $1/6$ (by the
Szemer\' edi-Trotter Theorem \cite{SzTr}).

Restricting to the case of complex numbers, Solymosi \cite[Lemma~1]{soly}
has proven a Szemer\' edi-Trotter-type result over $\C$ with $\delta=1/6$,
under the additional assumption that the set of points form a Cartesian
product in $\bb C^2$. Our result has a small $\delta$ but does not require
this additional assumption.  One would expect that $\delta=1/6$ holds
without any additional assumptions, and indeed, a tight result appears in
a paper on the Math ArXiv by Csaba D. T\'oth \cite{Toth2003}.

We conjecture that one can set $\delta=1/6$ in $\Zmodp$ given that $N$ is
sufficiently small compared to $p$. (This implies $\delta =1/6$ for the complex
case.)

\begin{proof}[Proof of Theorem~\ref{ourSzTr}]
Without loss of generality, assume that $\abs{\mc P} = \abs{\mc L} = N$,
adding ``dummy'' points and lines if necessary.  Say that $\mc P =
\{(x_i,y_i): i= 1,\ldots, N\}$, and, uniquely parameterizing a line $y = mx+b$
by the ordered pair $(m,b)$, say that $\mc L=\{ (m_i,b_i) : i=1,\ldots,N\}$.
Let $S:= \bigcup_{i=1}^N\{ x_i,y_i,m_i,b_i\}$,  set
\e{L_0&:=
\{x_i-x_j: 1\le i< j \le N\} \cup
\{y_i-y_j: 1\le i< j \le N\} \cup\\
&\qquad\cup\{m_i-m_j: 1\le i< j \le N\} \cup \{b_i-b_j: 1\le i< j
\le N\}, } and
let $L:=L_0\setminus \{0\}$.
By Theorem~\ref{redu2}, there exists a prime $q > N$ and a ring
homomorphism $\mapq: \Z[S] \to \Zmodq$ such that $0\notin\mapq(L)$.
Define a map $\Mapq: \Z[S]\cross \Z[S] \to \Zmodq \cross \Zmodq$ by
$\Mapq(a,b) = (\mapq(a), \mapq(b))$.  Because $0\notin\mapq(L)$, we
know that $\abs{\Mapq(\mc P)} = \abs{\Mapq(\mc L)} = N$. Thus, by
Theorem~\ref{t:bktSzTr}, there exist absolute constants $c$ and
$\delta$ such that \e{ \abs{\{(p',\ell')\in \Mapq(\mc P)\cross
\Mapq(\mc L): p'\in \ell'\} \rule{0pt}{12pt} } \le c N^{3/2 -\delta}. }
Since $\mapq$ is a ring homomorphism, the equation $y =mx+b$ implies that
$\mapq(y) = \mapq(mx+b) = \mapq(m)\mapq(x) + \mapq(b)$; and thus,
$$\abs{\{(p,\ell)\in \mc P\cross \mc L: p \in \ell\}} \le
\abs{\{(p',\ell')\in \Mapq(\mc P)\cross \Mapq(\mc L): p'\in \ell'\}} \le c
N^{3/2 -\delta},
$$
completing the proof.
\end{proof}

\section{A sum-product result for \czid s }\label{s:sum prod}

Given a subset $A$ of a ring, we define $A+A := \{a_1 + a_2: a_1,a_2 \in A\}$
and  $A A := \{a_1 a_2: a_1,a_2 \in A\}$.  Heuristically, sum-product
estimates state that one cannot find a subset $A$ such that both $A+A$ and $A
A$ have small cardinality, unless $A$ is close to a subring.  The first
sum-product result was proven in 1983 by Erd\H os and Szemer\' edi
\cite{ErSz} for the integers, and there have been numerous improvements and
generalizations, see for example \cite{Nat}, \cite{Ford}, \cite{Elekes}, and
\cite{Chang}.  Proving sum-product estimates in $\Zmodp$, where $p$ is a
prime, has been the focus of some recent work (see, for example, \cite{BKT},
\cite{B}, and \cite{Konyagin}), with the best known bound due to Katz and
Shen \cite{KS}, slightly improving a result of Garaev \cite{Garaev}:

\begin{theorem}[(\cite{KS})]
\label{garaevsumprod} Let $p$ be a prime and let $A$ be a subset of $\Zmodp$
such that $\abs{A} < p^{1/2}$. Then, there exist absolute constants $C$ and
$\alpha$ such that 
$$ C\abs A ^{14/13} (\log \abs A)^\alpha \le
\max\{\abs{A+A}, \abs{ AA}\}.$$
\end{theorem}

Theorem~\ref{oursumprod} demonstrates the same lower bound on $\max\{
\abs{A+A}, \abs{A A}\}$ for any finite subset $A$ of a \czid.

\begin{theorem} \label{oursumprod}
There are positive absolute constants $C$ and $\alpha$ such that, for every finite subset
$A$ of a \czid, $$C\abs{A}^{14/13} (\log \abs A)^\alpha \le \max\{ \abs{A+A},
\abs{AA}\}.$$
\end{theorem}
\noindent
The constants $C$ and $\alpha$ in this result is the same as those in
Theorem~\ref{garaevsumprod}.

Theorem~\ref{oursumprod} applies to a very general class of rings; however,
our mapping approach requires that the rings be commutative and have
characteristic zero.  For some results in the non-commutative case, see
\cite{Chang}; and for some results in $\Z/m\Z$ where $m$ is a composite, see
\cite{Chang2}.

\begin{proof}[Proof of Theorem~\ref{oursumprod}]  Because we are interested
in a lower bound on $\abs{A+A}$ and $\abs{AA}$, all we need in order to
apply Theorem~\ref{garaevsumprod} is a ring homomorphism $\phi$ from the
given \czid\ to $\Zmodp$ satisfying $\abs{\phi(A)} = \abs A$ (since any ring
homomorphism automatically satisfies $\abs{\phi(A) + \phi(A)} =
\abs{\phi(A+A)} \le \abs{A+A}$ and $\abs{\phi(A)\phi(A)} \le \abs{AA}$).
However, Theorem~\ref{redu2} also makes it easy to find a ring homomorphism
that preserves the cardinalities of $A+A$ and $AA$, as we will show below
(such a map would be useful for proving \emph{upper} bounds on $\abs{A+A}$
and $\abs{AA}$).  

Let
$$L_0 := \{a_1 - a_2: a_1,a_2 \in A\} \cup
\{a_1 + a_2 - (a_3 + a_4) : a_i \in A\} \cup
\{a_1 a_2 - a_3 a_4 : a_i \in A\}$$
and let $L:= L_0 \setminus \{0\}$.

By Theorem~\ref{redu2}, there exists a prime $p > \abs {A}^2$ and  a
ring homomorphism $\qmap: \Z[A]\to \Zmodp$ such that
\begin{enumerate}
\item[(i)] \qquad $\abs {\qmap (A)}= \abs { A} $, 
%
\item[(ii)] \qquad $ \abs {\qmap(A) + \qmap(A)} = \abs { A + A}$, and
%
\item[(iii)] \qquad $ \abs {\qmap(A)  \qmap(A)} = \abs { A  A}$.
\end{enumerate}
All three facts above follow from the definition of a ring homomorphism,
along with the definition of $L$ and the fact that $0 \notin \qmap(L)$.
We can now apply Theorem~\ref{garaevsumprod} to get that there exist
positive constants $C$ and $\alpha$ such that
\begin{equation*}
C\abs { \qmap(A)}^{14/13}(\log \abs A)^\alpha 
\le \max\{\abs { \qmap(A) + \qmap(A)} ,
\abs{ \qmap(A)  \qmap(A)}\}.
\end{equation*}
Finally, substituting (i),  (ii), and (iii) into this inequality
gives the desired result.
\end{proof}

\section{A matrix product result for $SL_2(D)$}\label{s:Helf}

In this section, we will consider finite subsets of the special linear group
$\SL_2(D)$ of $2$ by $2$ matrices with determinant 1 and entries in a \czid\
$D$.   For $A$ a finite subset of $\SL_2(D)$, let $\angles A$ denote the
smallest subgroup of $\SL_2(D)$ (under inclusion) that contains $A$.  We will
refer to $\angles A$ as the \emph{group generated by $A$}.
In general, the goal of this section will be to give conditions on $\angles A$
so that cardinality of the triple product $AAA:= \{a_1a_2a_3 : a_i \in A\}$ is
large.

Helfgott proved the following theorem in \cite{Helf}:
\begin{theorem}[(\cite{Helf})]\label{Helfthm}
Let $p$ be a prime. Let $A$ be a subset of $\SL_2(\Zmodp)$ not contained in
any proper subgroup, and assume that $\abs A < p^{3-\epsilon}$ for some
fixed $\epsilon>0$.  Then
$$\abs{AAA} > c \abs{A}^{1+\delta},$$
where $c >0$ and $\delta>0$ depend only on $\epsilon$.
\end{theorem}

\noindent
In this section, we will prove the following related result by combining
Theorem~\ref{Helfthm} with Theorem~\ref{redu2}.  A group $G$ is
\emph{metabelian} if $G$ has an abelian normal subgroup $N$
such that the quotient group $G/N$ is also abelian.

\begin{theorem}\label{ourHelf}
Let $A$ be a finite subset of $\SL_2(D)$, where $D$ is a \czid, and let
$\angles A$ be the subgroup generated by $A$.  If $\angles A$ has infinite
cardinality and $\angles A$ is not metabelian, then
$$\abs{AAA} > c \abs{A}^{1+\delta},$$
where $c>0$ and $\delta>0$ are absolute constants.
\end{theorem}

\noindent
One should note that Chang \cite{Chang3} has already proven a very similar
product result for $\SL_2(\bb C)$, in which ``metabelian'' is replaced by
``virtually abelian''.  A group $G$ is \emph{virtually abelian} if $G$
has a finite index subgroup $H$ such that $H$ is abelian.
\begin{theorem}[(\cite{Chang3})]\label{ChangHelf}
Let $A$ be a finite subset of $\SL_2(\bb C)$, and let $\angles A$ be the
subgroup generated by $A$.  If $\angles A$ is not virtually abelian (which
implies that $\angles A$ has infinite cardinality), then
$$\abs{AAA} > c \abs{A}^{1+\delta},$$
where $c>0$ and $\delta>0$ are absolute constants.
\end{theorem}

There are many groups that are both metabelian and
virtually abelian, for example all abelian groups satisfy both properties.
However, neither property implies the other.  For example, the group $G:=
\prod_{i=1}^\infty S_3$ (the product of infinitely many copies of the
symmetric group on three elements) is metabelian (since $N:=
\prod_{i=1}^\infty \angles{(123)}$ is an abelian, normal subgroup of $G$ such
that $G/N$ is also abelian), but $G$ is not virtually abelian.  On the other
hand, $G:= S_4 \times \Z$ is virtually abelian (since $H:= \angles{(1)} \times
\Z $ is a finite-index abelian subgroup of $G$), but $G$ is not metabelian
(since $S_4$ is not metabelian).

One major difference between Theorem~\ref{ourHelf} and Theorem~\ref{ChangHelf}
is in how the two results are proved.  Below, we will prove
Theorem~\ref{ourHelf} using Helfgott's Theorem~\ref{Helfthm}
as a black box along with some group theory and an easy application of
Theorem~\ref{redu2}.  On the other hand, Theorem~\ref{ChangHelf} is proven in
\cite{Chang3} by adapting Helfgott's methods in \cite{Helf} from the case of
$\SL_2(\Zmodp)$ to $\SL(\C)$ and using tools from additive combinatorics.

The constants $\delta>0$ in Theorems~\ref{ourHelf} and \ref{ChangHelf} are
not the best possible if one restricts to a subgroup.  For example,
$\SL_2(\Z)$ contains a subgroup isomorphic to $F_2$, the free group on 2
generators, and the following product result has recently been shown by
Razborov \cite{Raz}:
\begin{theorem}[(\cite{Raz})]
Let $A$ be a finite subset of a free group $F_m$ (on $m$ generators) with at
least two non-commuting elements.  Then,
$$\abs{AAA} \ge \frac{\abs{A}^2}{(\log\abs A)^{O(1)}}.$$
\end{theorem}

One should note that neither Theorem~\ref{ourHelf} nor
Theorem~\ref{ChangHelf} fully characterizes finite subsets of $\SL_2(\bb C)$
that have expanding triple product.  For example, neither theorem applies when
$A$ is contained in an abelian subgroup, but letting $$A:=\cbraces{
\parens{\begin{matrix} 1 & 2^j \\ 0 & 1\end{matrix}}: 1\le j \le n },$$ we have that
$\abs{AAA} \ge \abs{AA} = \binom{n+1}{2} > n^2/2 = \abs{A}^2/2.$ One should
also note that a sum-product theorem similar to
Theorem~\ref{oursumprod} does not hold in general for matrices.  As pointed
out in \cite[Remark 0.2]{Chang4}, the subset $$A:=\cbraces{ \parens{\begin{matrix} 1 &
j \\ 0 & 1\end{matrix}}: 1\le j \le n }$$ has the property that both the sumset
and product set are small:  $\abs{A+A} = \abs{AA} = 2n-1$.  However, it is
also shown by Chang \cite{Chang4} that by adding the assumption that the
matrices in $A$ are symmetric, one can prove a sum-product result similar to
Theorem~\ref{oursumprod}.

We now turn our attention to the proof of Theorem~\ref{ourHelf}.

\begin{proof}[Proof of Theorem~\ref{ourHelf}]
Say that $A$ is a finite subset of $\SL_2(D)$, where $D$ is a \czid.  Let $G
:= \angles A$, the subgroup generated by $A$, and assume that $G$ has infinite
cardinality and is not metabelian.  Let $T$ be the set of all normal subgroups
$N$ of $G$ such that $G/N$ is abelian (note that we include $G$ in the set
$T$), and define $$N_0:= \bigcap_{N\in T} N.$$  Then $N_0$ is a normal subgroup
of $G$ and $G/N_0$ is abelian.  Since $G$ is not metabelian by assumption, we
know that $N_0$ is not abelian, and so there exists $B_1,B_2 \in N_0$ such
that $B_1B_2 \ne B_2B_1$.  Also, let $M_1,M_2,M_3,\ldots, M_{121}$ be 121
distinct elements of $G$ (note $G$ is infinite by assumption).  We may now
define a set $L_0$ as follows:
\e{
    L_0 &:= \cbraces{b_i-c_j :
    \mbox{\parbox{3.2in}{$i,j \in \{1,2,3,4\}$ and
    $b_i$ and $c_j$ are entries in matrices $
    \parens{\begin{matrix} b_1& b_2\\ b_3 & b_4\end{matrix}},
    \parens{\begin{matrix} c_1& c_2\\ c_3 & c_4\end{matrix}} \in A$
    }}}\\[12pt]
& \qquad \cup \cbraces{b_i-c_j :
    \mbox{\parbox{4.5in}{$i,j \in \{1,2,3,4\}$ and
    $b_i$ and $c_j$ are entries in matrices \\ $
    \parens{\begin{matrix} b_1& b_2\\ b_3 & b_4\end{matrix}} \in M_{k_1}$ and
    $\parens{\begin{matrix} c_1& c_2\\ c_3 & c_4\end{matrix}} \in M_{k_2}$ for some
    $1\le k_1,k_2 \le 121$
    }}}\\[12pt]
& \qquad \cup \cbraces{b_1-1, b_2, b_3, b_4-1 :
    \mbox{\parbox{3.2in}{where
    $
    \parens{\begin{matrix} b_1& b_2\\ b_3 & b_4\end{matrix}} = B_1B_2B_1^{-1}B_2^{-1}
    \ne
    \parens{\begin{matrix} 1& 0\\ 0 & 1\end{matrix}}$
    }}}.
}
Let $L:=L_0\setminus \{0\}$, and let $S$ be the set of all entries that appear
in matrices in $A$.  By Theorem~\ref{redu2}, there exists $p > \abs
A$ and $\qmap: \Z[S] \to \Zmodp$  such that $0\notin \qmap(L)$.  Let $\Qmap:
\SLt(D) \to \SLt(\Zmodp)$ be defined by $\parens{\begin{matrix} b_1& b_2\\ b_3
& b_4 \end{matrix}} \mapsto\parens{\begin{matrix} \qmap(b_1)& \qmap(b_2)\\
\qmap(b_3) & \qmap(b_4) \end{matrix}}$.
Let $\uA:= \Qmap(A)$ and let $\uG:= \angles{\uA}$.  Note that by construction
$\abs A = \abs{\uA}$ and $\abs{AAA} \ge \abs{\uA\,\uA\,\uA}$, and also note that
$\abs{\uG} \ge 121$.

Assume for the sake of a contradiction that $\uG$ is a proper subgroup of
$\SLt(\Zmodp)$.  In \cite{Suz}, Suzuki gives the following classification of
the proper subgroups of $\SLt(\Zmodp)$:
\begin{theorem}[(cf. Theorem~6.17 of \cite{Suz}, page 404)]\label{suzthm}
Let $\uG$ be a proper subgroup of $\SLtp$ where $p\ge 5$.  Then $\uG$ is
isomorphic to one of the following groups (or to a subgroup of one of the
following groups):
\begin{romanlist}
\ii a cyclic group,

\ii the group with presentation $\angles{x,y \big| x^m=y^2, y^{-1}xy =
x^{-1}}$, which has order $4m$,

\ii a group $H$ of order $p(p-1)$ having a Sylow-$p$ subgroup $Q$ such that
$H/Q$ is cyclic and $Q$ is elementary abelian,

\ii the special linear group $\SLt(\Z/3\Z)$ on a field of three elements,
which has order 24,

\ii $\widehat S_4$, the representation group of $S_4$ (the symmetric group on 4
letters), which has order 48, or

\ii the special linear group $\SLt(\Z/5\Z)$ on a field of five elements,
which has order 120.

\end{romanlist}
\end{theorem}

Since $\abs{\uG} > 120$, we may eliminate (iv), (v), and (vi) as
possibilities.
The remaining possibilities (namely, (i), (ii), and (iii)) are all metabelian;
and thus, $\uG$ must have a normal subgroup $\uN$ such that $\uN$ is abelian
and $\uG/\uN$ is also abelian.

Let $N:= \Qmap^{-1}(\uN)$.  Then $N$ is a normal subgroup of $G$, and by the
third isomorphism theorem $G/N \isom (G/\ker(\Qmap))/(N/\ker(\Qmap)) \isom
\uG/\uN$, which is abelian.  Thus, $N_0$ is a subgroup of $N$, and so $B_1,
B_2 \in N$.  We know that $B_1B_2B_1^{-1}B_2^{-1} \ne
\parens{\begin{matrix} 1&0\\0&1\end{matrix}}$, and by the definition of
$\Qmap$, we also have that
$$\Qmap(B_1)\Qmap(B_2)\Qmap(B_1)^{-1}\Qmap(B_2)^{-1} =
\Qmap(B_1B_2B_1^{-1}B_2^{-1}) \ne \parens{\begin{matrix}
1&0\\0&1\end{matrix}}.$$
But, this contradicts the fact that $\uN$ is abelian.  Thus, the assumption
that $\uG$ is a proper subgroup of $\SLtp$ is false, and we have that
$\angles{\uA} = \uG = \SLtp$.

Finally, by Theorem~\ref{Helfthm}, there exist absolute
constants $c>0$ and $\delta>0$ such that
$$\abs{AAA} \ge \abs{\uA\,\uA\,\uA} \ge c\abs{\uA}^{1+\delta} =
c\abs{A}^{1+\delta}.$$
\end{proof}

Another way to show that $\Qmap(A)$ generates all of $\SLtp$
would be to assume that $\angles A$ is not virtually solvable, which implies
by Tits Alternative Theorem \cite{Tits} that $\angles A$ has a non-abelian
free subgroup.  Then, following \cite[Section 2]{Gamb}, it is possible to
bound the girth of a certain Cayley graph from below in terms of $p$,
eventually showing (via an appeal to Theorem~\ref{suzthm})
that $\angles{\Qmap(A)}= \SLtp$.

Also, the proof above uses the following implicit corollary of
Theorem~\ref{suzthm}:  if $\uG$ is a proper subgroup of $\SLtp$ and
$\abs{\uG}>120$, then $\uG$ in metabelian.  A very similar result for
$\operatorname{PSL}_2(\Zmodp) \isom \SLtp/(\pm I)$ (where $I$ is the identity
matrix)
appears in \cite[Theorem~3.3.4, page 78]{Valette}.

\section{Random matrices with entries in a \czid } \label{s:mat}

In \cite{KKSz,TV}, it is shown that a random Bernoulli matrix of
size $n$ is singular with probability $\exp(-\Omega (n))$. One may
ask what happens for random matrices with complex entries.  We are
going to give a quick proof of the following:

\begin{theorem} \label{theorem:matrix}
For every positive number  $\rho <1$, there is a positive number
$\delta<1$ such that the following holds. Let $\xi$ be a random
variable with finite support in a \czid, where $\xi$ takes each
value with probability at most $\rho$.  Let $M_n$ be an $n$ by $n$
random matrix whose entries are iid copies of $\xi$. Then the
probability that $M_n$ is singular is at most $\delta ^n$, for all $n$
sufficiently large with respect to $\rho$ and the size of the support of $\xi$.
\end{theorem}

\begin{remark}In the case when the \czid\ is $\bb C$, more quantitative
bounds are available (see \cite{BVW,TVunpub}).
\end{remark}

Theorem~\ref{theorem:matrix} follows directly from the following two
results.

\begin{theorem} \label{theorem:matrixfield}
For every positive number  $\rho <1$, there is a positive number
$\delta <1$ such that the following holds. Let $n$ be a large positive
integer and $p \ge 2^{n^n}$ be a prime. Let $\xi$ be a random
variable with finite support in $\Zmodp$, where $\xi$ takes each
value with probability at most $\rho$.  Let $M_n$ be an $n$ by $n$
random matrix whose entries are iid copies of $\xi$. Then the
probability that $M_n$ is singular is at most $\delta ^n$, for all $n$
sufficiently large with respect to $\rho$ and the size of the support of
$\xi$.
\end{theorem}

This theorem was implicitly proved in \cite{TV}. The bound $2^{n^n}$
is not essential, we simply want to guarantee that $p$ is much larger
than $n$. The reason that the proof from \cite{TV} does not extend
directly to the complex case (or \czid s in general) is that in
\cite{TV} one relied on the identity

$$ \I_ {x=0} = \int_0^1 \exp(2\pi i xt) dt, $$

\noindent where $\I$ is the indicator function.  This identity holds for $x$
an integer, but it is not true for complex numbers in general.
Theorem~\ref{redu2} provides a simple way to overcome this obstacle. (For
other methods, see \cite{TVunpub,TVcir}.)

\begin{lemma}\label{reduction theorem}
Let $S$ be a finite subset of a characteristic zero integral domain.  There
exist arbitrarily large primes $p$ such that there is a ring homomorphism
$\qmap:\Z[S] \to \Zmodp$ satisfying the following two properties:
\begin{enumerate}
\item[(i)]  the map $\qmap$ is injective on $S$, and
\item[(ii)]  for any $n$ by $n$ matrix $(s_{ij})$ with entries $s_{ij}\in S$,
we have $$\det(s_{ij}) = 0 \quad\mbox{ if and only if } \quad
\det\left(\qmap(s_{ij})\right)=0.$$
\end{enumerate}
\end{lemma}

\begin{proof}
Let $ L := \{ \det(s_{ij}): s_{ij}\in S\}\setminus \{0\}$.
Applying Theorem~\ref{redu2} gives us a ring homomorphism $\qmap: \Z[S] \to
\Zmodp$ (for some arbitrarily
large prime $p$) such that $0\notin\qmap(L)$.  Since $\qmap$ is a ring
homomorphism, $\qmap(\det(s_{ij})) = \det(\qmap(s_{ij}))$ and also $\qmap(0)=
0$; thus, we have satisfied condition (ii).

In this particular case, we will show that (i) follows from (ii).  If $S$
contains more than one element, we can find $s \ne t \ne 0$ both lying in
$S$, and thus
\[
\det\lt(\parens{\begin{matrix}
s & t & \cdots &t& t \\
t & s & t &\cdots & t \\
\vdots &t & \ddots &t& \vdots\\
t &\cdots &t & s& t\\
t &t& \cdots & t& t
\end{matrix}}\rt)
=
\det\lt(\parens{\begin{matrix}
s-t & 0 & \cdots &0& 0 \\
0 & s-t &0& \cdots & 0 \\
\vdots &0 & \ddots &0& \vdots\\
0 &\cdots &0& s-t& 0\\
0 &0&\cdots & 0& t
\end{matrix}}\rt)
= (s-t)^{n-1} t \ne 0.
\]
Thus, by (ii), $0 \ne
\left(\qmap(s)-\qmap(t)\right)^{n-1}\qmap(t)$, and so $\qmap(s)\ne
\qmap(t)$ and we see that $\qmap$ is injective on $S$.
\end{proof}

The fact that (ii) happens to imply injectivity on $S$ is not
important---in fact, for any given finite subset $A\subset \Z[S]$ we can find
$\phi_{\widetilde Q}$ satisfying (ii) above that is also injective on $A$ by
adding $\{a_1-a_2: a_1\ne a_2 \mbox{ and } a_1, a_2\in A\}$ to $\widetilde L$
in the proof above.  For example, we could find $\phi_{\widetilde Q}$ that is
injective on the set of all determinants of $n$ by $n$ matrices with entries
in $S$.

One should note that it is easy to prove results similar to
Lemma~\ref{reduction theorem}  where the determinant is replaced by some
polynomial $f(x_1,x_2,\ldots,x_k)$ with integer coefficients and one wants
a map $\qmap$ such that $f$ evaluated at points in $S$ is zero if and only if
$f$ evaluated at points in $\qmap(S)$ is zero.  This can also easily be
extended to the case where $f$ is replaced by a list of polynomials, each of
which is evaluated on some subset of $S$.

\section{The density theorem}

The number 7 is a prime in the ring of integers $\bb Z$; however, if one
extends $\Z$ to $\Z[\sqrt 2]$, the prime 7 splits: $7 = (3 -\sqrt 2)(3 + \sqrt
2)$.  This fact has the same mathematical content as the following: the
polynomial $x^2-2$ is irreducible in $\Z[x]$; however, in $(\Z/7\Z)[x]$, where
the coefficients of the polynomial are viewed as elements of $\Z/7\Z$, the
polynomial splits: $x^2-2 = (x - 3)(x + 3)$.  The Frobenius Density Theorem
describes how frequently such splitting occurs.  In modern formulations, the
Frobenius Density Theorem quantifies the proportion of primes that split in a
given Galois extension of the rational numbers.  We will use the following
historical version given in \cite[page 32]{Stev}, which is phrased in terms of
polynomials splitting modulo $p$.  Note that the relative density of a set of
primes $S$ is defined to be $$\lim_{x\to\infty} \frac{\abs{\{p\le x: p\in
S\}}}{\abs{\{p\le x: p\mbox{ is prime}\}}}.$$

\begin{theorem}[(Frobenius Density Theorem)]\label{frob}
Let $g(z) \in \Z[z]$ be a polynomial of degree $k$ with $k$ distinct roots in
$\bb C$, and let $G$ be the Galois group of the polynomial $g$, viewed as a
subgroup of $S_{k}$ (the symmetric group on $k$ symbols).  Let
$n_1,n_2,\ldots, n_t$ be positive integers summing to $k$.  Then, the relative density
of the set of primes $p$ for which $g$ modulo $p$ has a given decomposition
type $n_1,n_2,\ldots, n_t$ exists and is equal to $1/\abs{G}$ times the number
of $\sigma \in G$ with cycle pattern $n_1,n_2,\ldots,n_t$.
\end{theorem}
\noindent
For example, since the identity element corresponds to the cycle pattern
$1,1,\ldots,1$ and every group has one identity, the relative density of primes $p$ such
that $g$ decomposes into $k$ distinct linear factors modulo $p$ is
$1/\abs{G}$.

Theorem~\ref{frob} is the version proven by Frobenius in 1880 and published in
1896 \cite{dasfrob}.  In \cite{Stev}, Stevenhagen and Lenstra give numerous
examples and an illuminating discussion of the original motivation for the
Frobenius Density Theorem and how it relates to the stronger Chebotarev
Density Theorem.


\section{Proof of Theorem~\ref{redu2}} \label{s:proofs}

The first step towards proving Theorem~\ref{redu2} is proving the
following lemma.

\begin{lemma}
\label{redu1}
Let $S$ be a finite subset of a characteristic zero integral domain $D$, and
let $L$ be a finite set of non-zero elements in the subring $\Z[S]$ of $D$.
Then there exists a complex number $\theta$ that is algebraic over $\bb Q$ and
a ring homomorphism $\phi: \Z[S]\to\Z[\theta]$ such that $0\notin \phi(L) $.
\end{lemma}

By itself, this lemma allows one to extend sum-product and incidence problem
results proven in the complex numbers to any \czid\ (in much the same way that
Theorem~\ref{redu2} allows one to extend such results proven in $\Zmodp$ to any
\czid).

Lemma~\ref{redu1} is proved using three main steps: applying the primitive
element theorem, applying Hilbert's Nullstellensatz to pass to the case of
only algebraic numbers, and applying the primitive element again to get to a
ring of the form $\Z[\theta]$.  Each of these three steps requires negotiating
between the rings we are interested in and their fraction fields.
Theorem~\ref{redu2} is proved by combining Lemma~\ref{redu1} with the
Frobenius Density Theorem (or the stronger Chebotarev Density Theorem) to pass
to a quotient isomorphic to $\Zmodp$.

\begin{remark}[An effective version of
Theorem~\ref{redu2}]\label{RemEffective}
It would be interesting to prove a version of Theorem~\ref{redu2} that
included an upper bound on at least one (or more) of the primes $p$ (in
terms of $S$ and $L$) for which desired homomorphism $\quotientmap$ exists.
One possible program for proving such a result would be to follow the
general outline of the proof of Theorem~\ref{redu2} given in this section,
combined with effective versions of the primitive element theorem, Hilbert's
Nullstellensatz, and the Chebotarev Density Theorem.
\end{remark}

\begin{proof}[Proof of Lemma~\ref{redu1}]
Let $S$ be a finite subset of a characteristic zero integral domain $D$.
Recall that we identify the subring of $D$ generated by the identity with $\Z$
and so we use $\Z[S]$ to denote the smallest subring of $D$ containing $S$.

We can write $S=\{x_1,x_2,\ldots,x_j,\theta_1,\theta_2,\ldots,\theta_k\}$,
such that $\{x_1,x_2,\ldots,x_j\}$ are independent transcendentals over $\Q$
and such that $K$, the fraction field of $\Z[S]$, is algebraic over
$\Q(x_1,x_2,\ldots,x_j)$.
Using the primitive element theorem, we can find $\widetilde \theta$ in $K$
also algebraic over $\Q(x_1,x_2,\ldots,x_j)$ such that
$$\Q(x_1,x_2,\ldots,x_j,\theta_1,\theta_2,\ldots,\theta_k) = 
\Q(x_1,x_2,\ldots,x_j,\widetilde \theta).
$$
To get the analogous statement for $\Z$ instead of $\Q$, 
we write, for each $i$ $$\theta_i=\sum_k \frac{f_{i,k}}{g_{i,k}}\widetilde{\theta}^k,$$ 
where $f_{i,k},g_{i,k}\in \Z[x_1,x_2,\ldots,x_j]$,
and
we then define $\theta_0$ to be $\widetilde \theta$
divided by the
product of the $g_{i,k}$.
Thus, we can find $\theta_0$ in $K$ also algebraic over
$\Q(x_1,x_2,\ldots,x_j)$ such that
$$\Z[S]\subset \Z[x_1,x_2,\ldots,x_j,\theta_0]\isom
\Z[y_1,y_2,\ldots,y_{j+1}]/f_0,$$ where the $y_i$ are formal variables and
$f_0$ is an irreducible element in $\Z[y_1,y_2,\ldots,y_{j+1}]$ that is
non-constant or zero and that gives zero when evaluated at $y_i = x_i$ for
$i=1,\ldots,j$ and $y_{j+1}=\theta_0$.

Let $\Qbar$ be the algebraic closure of the rational numbers, let $\mc L':=
\prod_{\ell \in L} \ell$, and let $\mc L \in \Z[y_1,\ldots, y_{j+1}]$ be the
lowest degree representative of the image of $\mc L'$ under the above
inclusion and isomorphism.  We will use the following corollary to Hilbert's
Nullstellensatz:

\begin{proposition}[c.f. the corollary on page 282 of \cite{Shaf}]
\label{nullstell}
If $\mc L, f_0 \in \Qbar[y_1,\ldots, y_{j+1}]$ and if on points of
$\Qbar^{j+1}$ we have that $\mc L$ is zero whenever
$f_0$ is zero, then there exists $m \ge 0$ and $k\in
\Qbar[y_1,\ldots,y_{j+1}]$ such that $\mc L^m = k f_0$.
\end{proposition}

Say that $\mc L^m = k f_0$ for some $k \in \Qbar[y_1,\ldots,y_{j+1}]$.  Since
$\mc L^m, f_0 \in \Zr$, we have that $k$ is in $\Q(y_1,\ldots,y_{j+1})$ (the
fraction field of $\Zr$).  Thus, $k$ is in the ring $\Q[y_1,\ldots,y_{j+1}]$,
and so there is a positive integer $c$ such that $ck \in \Zr$.  We now have
$c\mc L^m = (ck) f_0 $.
Since $f_0$ is irreducible in $\Z[y_1,y_2,\ldots,y_{j+1}]$, we must have that
$f_0$ divides $\mc L$ ($f_0$ cannot divide the positive integer $c$ since
$f_0$ is either non-constant or zero).  But this is impossible since by
assumption, $\mc L$ is non-zero in the quotient ring $\Zr/f_0$.  Thus, for
every $m \ge 0$ and for every $k\in \Qbar[y_1,\ldots,y_{j+1}]$ we must have
that $\mc L^m \ne k f_0$.
%
Therefore, by the  contrapositive  of Proposition~\ref{nullstell}, there exist
algebraic numbers $q_1, \ldots, q_{j+1} \in \Qbar $ such that
%
$f_0\Big|_{y_i=q_i}=0$ while $\mc L\Big|_{y_i=q_i} \ne 0$.
Thus, we have a homomorphism $$ \psi_0:
\Z[y_1,y_2,\ldots,y_{j+1}]/f_0 \to \Z[q_1, \ldots, q_{j+1}],$$ defined by $y_i
\mapsto  q_i$, such that $ \psi_0( \mc L ) \ne 0.$

Applying the  primitive  element theorem and clearing denominators as before, we have
$$\Z[q_1, \ldots, q_{j+1}] \subset \Z[\theta_1],$$
with $\theta_1\in\Qbar$.
Combining the inclusions and isomorphisms from the applications of the primitive
element theorem with $\psi_0$ completes the proof of Lemma~\ref{redu1}.
\end{proof}

Recall the statement of Theorem~\ref{redu2}:
\setcounter{mirrorcount}{1}
\begin{mirrorprop}
Let $S$ be a finite subset of
a characteristic zero integral domain $D$, and let $L$ be a finite set of
non-zero elements in the subring $\Z[S]$ of $D$.  There exists an infinite
sequence of primes with positive relative density such that for each prime $p$
in the sequence, there is a ring homomorphism $\quotientmap: \Z[S]\to\Zmodp$
satisfying $0\notin \quotientmap(L)$.
\end{mirrorprop}
The proof of Theorem~\ref{redu2} picks up where the proof of Lemma~\ref{redu1} left off.

\begin{proof}[Proof of Theorem~\ref{redu2}]
%
By Lemma~\ref{redu1}, there exists a ring homomorphism $$\phi: \bb Z[S] \to
\Z[\theta_1] \isom \Z[z]/f_1,$$ such that $0\notin\phi(L)$, where $z$ is a
formal variable and $f_1$ is an irreducible element in $\Z[z]$ that gives zero
when evaluated at $z=\theta_1$.

Let $\widehat{L} := \prod_{\ell \in L} \ell$, let
$\widetilde{L}(z)\in \Z[z]$ denote the lowest-degree representative of
$\phi(\widehat L)$ in $\Z[z]/f_1$, and let $L_1(z)$ denote the product of all
distinct irreducible factors of $\widetilde{L}(z)$ in $\Z[z]$. Note that a
homomorphism of integral domains will map $\widetilde{L}(z)$ to zero if and
only if it maps $L_1(z)$ to zero.  By assumption, $\widetilde{L}(z)$ is
non-zero, so we must have that $f_1(z)$ does not divide $\widetilde{L}(z)$ in
$\Z[z]$; and thus $f_1(z)$ does not divide $L_1(z)$.
Therefore, $L_1(z)$ has no roots (in $\bb C$, say) in common with $f_1(z)$,
since $f_1(z)$ is irreducible.

By Theorem~\ref{frob} (the Frobenius Density Theorem)
there exists a sequence of primes $(p_1,p_2,p_3,\ldots)$ in $\Z$ (with positive
relative density)
such that for any prime $p$ in the sequence, the polynomial $f_1(z) L_1(z)$
factors completely modulo $p$ into a product of $\deg\lt(f_1(z) L_1(z)\rt)$
distinct linear factors.

Let $(z-a)$ be a linear factor of $f_1(z)$ modulo $p$, where $p$ is any prime
in the sequence $(p_1, p_2, p_3,\ldots)$.
Since, modulo $p$, the linear factors of $f_1(z)$
are all distinct from those of $L_1(z)$, we know that $(z-a)$ does not
divide $L_1(z)$ modulo $p$.  Thus, for infinitely many primes $p$, we may
quotient out by $p$ and by $(z-a)$ to get a canonical quotient map
$$\psi_1: \Z[z]/f_1 \longrightarrow \Z[z]/(p,z-a) \isom \Zmodp $$
where $\psi_1(L_1(z)) \ne 0$.
One can think of $\psi_1$
as modding out by $p$ and then sending $z$ to the
element $a$ in $ \Zmodp$.

Letting $\qmap := \psi_1\circ \phi$
completes the proof.
\end{proof}

\section{Acknowledgments}

We would like to thank Ellen Eischen for providing numerous clarifications and
simplifications on an early version of Theorem~\ref{redu2} and its proof.
Thanks is also due to John Bryk, Wei Ho, Peter Stevenhagen, Terence Tao, and
J.B. Tunnell for useful comments.  Finally, we would like to thank the
anonymous referee for the many exceptionally careful and useful comments on
the manuscript, all of which have improved the paper.

\end{document}